\documentclass[reqno]{amsart}
 \usepackage{amssymb,amsmath,amscd,graphicx,color,epstopdf,mathtools,comment}
\usepackage[all]{xy}
\usepackage[all]{xy}
\usepackage{color}
\usepackage{float}
\usepackage{hyperref}
\usepackage{enumerate}
\usepackage{amsthm}
\usepackage{epsfig}
\usepackage[english]{babel}
\usepackage[latin1]{inputenc}

\newcommand{\RR}{{\mathbb{R}}}

\renewcommand{\gg}{\mathfrak{g}}

\newcommand{\nn}{\mathfrak{n}}

\newcommand{\kk}{\mathfrak{k}}
\renewcommand{\ll}{\mathfrak{l}}

\renewcommand{\aa}{\mathfrak{a}}
\newcommand{\uu}{\mathfrak{u}}

\renewcommand{\ss}{\mathfrak{s}}

\renewcommand{\sl}{\mathfrak{sl}}
\newcommand{\so}{\mathfrak{so}}
\renewcommand{\sp}{\mathfrak{sp}}

\newcommand{\cC}{\mathcal{C}}

\newcommand{\cO}{\mathcal{O}}

\newcommand{\cS}{\mathcal{S}}

\newcommand{\cU}{\mathcal{U}}

\newcommand{\R}{\mathbb{R}}

\newcommand{\C}{\mathbb{C}}

\newcommand{\B}{\mathrm{B}}
\newcommand{\E}{\mathrm{E}}

\newcommand{\pr}{\mathrm{pr}}

\newcommand{\Sl}{\mathrm{Sl}}
\newcommand{\Gl}{\mathrm{Gl}}
\newcommand{\So}{\mathrm{So}}
\newcommand{\Sp}{\mathrm{Sp}}

\renewcommand{\L}{\mathrm{L}}

\newcommand{\U}{\mathrm{U}}
\newcommand{\G}{\mathrm{G}}

\newtheorem{proposition}{Proposition}

\newtheorem{theorem}{Theorem}

\theoremstyle{definition} 
\newtheorem{example}{Example}
\newtheorem{remark}{Remark}

\title[Diagonalization of the Toda flow]{Diagonalization of the Toda flow  for \\ arbitrary isospectral symmetric matrices}
	 \author{David Mart\'inez Torres}
	 \address{Department of Applied Mathematics, ETSAM Section, Universidad
	 	Polit\'ecnica de Madrid,
	 	Avda. Juan de Herrera 4, 28040 Madrid, Spain}
	 \email{df.mtorres@upm.es}

	  \author{Carlos Tomei}
	 \email{carlos.tomei@mat.puc-rio.com}

\begin{document}

\begin{abstract} We construct  coordinates on orthogonal conjugacy classes of traceless real symmetric  matrices with arbitrary spectrum  (the isospectral manifold) that diagonalize the non-periodic Toda vector field. The coordinates, defined on a neighborhood of any diagonal matrix decouple the Toda vector field into a sum of multiples of the Euler vector field in $\R$. The domain of each set of coordinates is dense and  their union covers the isospectral manifold.  The construction relies on a new matrix factorization  which is of independent interest.
 \end{abstract}
\maketitle

\section{Introduction}

\medskip

 The celebrated Toda lattice was conceived originally  \cite{T} as a Hamiltonian on a chain of $n$ particles connected to neighbors by specific nonlinear springs\footnote{For  historical information concerning the Toda flow, see \cite{DLSTT}.}. A change of variables introduced by Flaschka  \cite{Fl} transformed the resulting dynamics into a vector field  on Jacobi matrices given by a Lax pair,
\begin{equation}\label{eq:Toda}
	X'=[X,\pi_\so X],\quad \mathrm{I}=\pi_\so+ \pi_{\uu},
\end{equation}
where the identity map $\mathrm{I}$ on $\mathfrak{sl}$ is decomposed as the sum of projections
onto traceless skew-symmetric matrices $\so$, and traceless upper triangular matrices $\uu$. Flaschka realized that the set of Jacobi matrices is invariant under the flow, as is the  isospectral manifold
\[\cO=\{ Q^T X Q\,|\, Q \in \So\}.\] This essentially yields the complete integrability of the Toda lattice for an appropriate symplectic structure on Jacobi matrices of zero trace. For  action variables given by the eigenvalues of the initial condition, Moser \cite{Mo} obtained angle variables --the so-called norming constants.

The form of  equation \eqref{eq:Toda} naturally leads to the consideration of larger phase spaces (\cite{DLNT1, DLT2}), in which integrability was also proved. More recently \cite{MT}, for
the Toda vector field on any semisimple Lie algebra, coordinates were constructed on dense domains of analogs of orthogonal conjugacy classes of real symmetric matrices, under a genericity condition corresponding to the simplicity of the spectrum of the initial condition.  In \cite{MT2} the symmetry requirement was removed (see also \cite{LMST} for the case of real matrices).

The main purpose of this text is to remove a systematic obstruction: We handle real symmetric matrices,  with possibly multiple eigenvalues. This is a wide generalization in terms of phase spaces. In the simple spectrum case the orthogonal conjugacy class of real symmetric matrices is the manifold of real full flags in Euclidean space,
\[\{(F_1,F_2, \dots ,F_{n-2},F_{n-1})\,|\,F_i\subset\R^n,\,\,F_i\subset F_{i+1},\,\,\mathrm{dim} F_i=i\}.\]
Allowing arbitrary spectrum means that the orthogonal conjugacy class of real symmetric matrices can be any manifold of partial flags,
\[\cO=\{(F_{d_1},\dots,F_{d_k})\,|\,F_{d_i}\subset\R^n,\,\,F_{d_i}\subset F_{d_{i+1}},\,\,\mathrm{dim} F_{d_i}=d_i\}.\]
For instance, if there are just two eigenvalues with multiplicities $d$ and $n-d$, then $\cO$ is the Grassmannian of $d$-planes in $\R^n$. For $d=1$ we have real projective space as phase space.

Our main result shows that the Toda vector field on any orthogonal conjugacy class of real symmetric matrices is obtained by juxtaposing diagonal linear vector fields defined on covering local coordinates.
To describe the result in more detail, one first observes that the equilibria of the Toda vector field \eqref{eq:Toda} in $\mathcal{O}$ are the traceless real  diagonal matrices in there. For each such matrix $\Lambda$ with entry $\lambda_i$ in  position $(ii)$, 
we define the following affine subspace of the space $\L$ of  unit real lower triangular  matrices,

\begin{equation}\label{eq:chart-domain}
 \L^\Lambda=\{L\in \L\,|\, L_{ij}=0,\,\,\mathrm{if}\,\,\lambda_i=\lambda_j,\,i\neq j\}.
\end{equation}
 Note that $\L^\Lambda$ is not a subgroup in general. 
 Moreover,
 for any traceless real diagonal matrix $D$ the (complete) linear vector field $X'=[X,D]$, $X\in \sl$, is tangent to $\L^\Lambda$ and is diagonal: each coordinate $X_{ij}$ evolves independently of the others.

\begin{theorem}\label{thm:main} Let  $\cO$ be an orthogonal conjugacy class of real symmetric matrices.
	Then around each diagonal matrix $\Lambda\in \cO$ there exist local coordinates mapping onto $\L^\Lambda$
that transform the Toda vector field in the  diagonal linear  vector field
  \begin{equation}\label{eq:Toda-coordinates}
  L'=[L,-\Lambda],\quad L\in \L^\Lambda.
  \end{equation}

The domain of definition of each set of local coordinates is dense, and the union over the (finitely many) diagonal matrices in $\cO$ of these domains covers $\cO$.
\end{theorem}

 In other words, the evolution decouples in the different coordinates of $L$,
 \[L'_{ij}=(\lambda_i-\lambda_j)L_{ij}.\]
 
 As we shall see, the choice of $\L^\Lambda$ removes the ambiguity in the description as conjugates of elements of $\cO$ in a neighborhood of $\Lambda$, 
 \[ Q^T \Lambda Q,\quad Q=Q(L),\quad L\in \L^\Lambda \, ,\] 
 which is caused by the abundance of matrices commuting with the diagonal matrix $\Lambda$ with non-simple spectrum. 
 
 Even in the case of Grassmanians, the atlas granted by Theorem \ref{thm:main} is new.

%

\medskip
Theorem \ref{thm:main} is   known in the simple spectrum case \cite {LMST}, where the target of the coordinates is the group $\L$. As it turns out, the chart (inverse of the coordinates) used in the simple spectrum case $\Psi:\L\longrightarrow\cO$ already provides a diagonal linear lift for the Toda vector field around a diagonal matrix $\Lambda\in \cO$, regardless of whether its spectrum is simple or not. However, if the spectrum is not simple the dimension of $\L$ exceeds the dimension of $\cO$. The proof of Theorem \ref{thm:main} hinges on proving  that the restriction of $\Psi$ to $\L^\Lambda$ is a diffeomorphism onto its image. This is done by means of a constrained factorization result, which we find of independent interest.
To state it we introduce the subgroup of orthogonal real matrices which commute with $\Lambda$,
\[\So_\Lambda=\{Q\in \So\,|\, Q^T\Lambda Q=\Lambda\}.\]
 We denote by $\U\subset \Sl$ the group of determinant one upper triangular real matrices with positive diagonal entries.
\begin{theorem}\label{thm:constrained-QLU} There exist a unique factorization depending smoothly on parameters
 \[\So_\Lambda \L^\Lambda  \U\subset \Sl,\]
 valid for those matrices $X\in \Sl$ in the complement of the solution set of a polynomial equation in the entries of $X$.
\end{theorem}

 As $\Lambda$ ranges over diagonal matrices in $\cO$, Theorem \ref{thm:constrained-QLU} produces as many different factorizations as the $n$-th Bell number, where $n$ is the size of $\Lambda$.
Two of these cases are standard.
\begin{itemize}
 \item If $\Lambda$ has a unique eigenvalue, then $\So_\Lambda=\So$, $\L^\Lambda=\{\mathrm{I}\}$, and we recover the $\mathrm{QR}$ factorization, valid everywhere in $\Sl$.
 \item If $\Lambda$ has simple spectrum, then $\So_\Lambda$ are  (diagonal) sign matrices, $\L^\Lambda=\L$, and  we recover the (signed) $\L\U$ factorization, valid for matrices with nonzero principal minors.
\end{itemize}


Theorems \ref{thm:main} and \ref{thm:constrained-QLU} also hold for matrices  with complex or quaternionic entries, where the special orthogonal group is replaced by the special unitary group.

\subsection*{ Acknowledgments} The authors thank Froilan Dopico, Ricardo Leite and Nicolau Saldanha for extensive conversations,  and Qiyuan Gu for kindly providing  the counterexample in Section \ref{ssec:Lie-theory}. DMT and CT acknowledge  finantial support by CNPq and FAPERJ. DMT and CT also acknowledge respectively partial finantial support  by MCIN-AEI grant PID2022-139069NB-I00, and by StoneLab.

\section{Factorizations and diagonalizing coordinates}\label{sec:cx}

All matrices in this section are real.

We start by reminding the reader basic properties of the $\mathrm{QR}$ and the $\L\U$ factorizations of matrices, paying attention to  differential topology aspects.
Then we show how to extract from these factorizations the coordinates in Theorem \ref{thm:main} by means of a constrained $\mathrm{Q}\L\U$ factorization.

\medskip

 In  matrix theory a factorization in a subgroup $\mathrm{G}$ of the general linear group with factors $\mathrm{A},\mathrm{B}\subset \mathrm{G}$ say,  is
 a procedure that expresses a matrix $G$ in an appropriate subset of $\mathrm{G}$ as
 \[G=AB,\quad A\in \mathrm{A},\,B\in \mathrm{B}.\]
We are concerned with factorizations with the following additional properties: uniqueness and smooth dependence on parameters. By this we mean that
 \begin{itemize}
  \item  if $G=AB$ is given, the factorization returns the factors $A,B$;
  \item the domain of definition $\mathrm{AB}\subset \mathrm{G}$ is open and $\mathrm{A}$, $\mathrm{B}$ are (embedded) submanifolds of $\mathrm{G}$.
 \end{itemize}

  These properties can be restated in the language of differential topology: There exists submanifolds $\mathrm{A}, \B$, typically subgroups, such that the product map
\[\mathrm{A}\times \B \longrightarrow \G,\quad (A,B)\mapsto AB,\]
is a diffeomorphism onto its image $\mathrm{AB}\subset \G$ which is an open subset, typically the entire group $\G$. The factorization procedure is the explicit construction of the inverse to the product map.

Factorizations may have more than two factors. For example, the factorization in the statement of Theorem \ref{thm:constrained-QLU} has three factors and is not surjective.

A standard  factorization in the special linear group is the $\mathrm{QR}$ factorization
\[\So \U=\Sl,\]
where $\U$ denotes the group of determinant one upper triangular matrices  with positive diagonal entries. The factorization is obtained, for instance, from the Gram-Schmidt algorithm.

We also consider the $\L\U$ factorization in the special linear group. Its domain of definition is the subset of matrices with strictly positive principal minors,
\[\L\U\subset \Sl.\]
We are  interested in the intersection of the image $\L\U$ with $\So$, in the open subset $\cC\subset \So$ of special orthogonal matrices with strictly positive principal minors,
\[\cC=\{Q\in \So\,|\,\exists L\in \L,\,U\in \U,\, Q=LU\}.\]
We will be using a  simple property of  $\cC$ (see \cite[Corollary 1]{MT}): $\L$ is  diffeomorphic to $\cC$ by means of  the restriction
to $\L$ of the first projection in the $\mathrm{QR}$ factorization $\kappa:\Sl\to \So$,

\begin{equation}\label{eq:first-projection}\kappa:\L\longrightarrow \cC.
 \end{equation}

Let $\Lambda$ be a traceless diagonal matrix and let $\cO$ be its orthogonal conjugacy
class, i.e., the isospectral manifold consisting of (real) symmetric matrices
\[\cO=\{Q^T\, \Lambda\,  Q\,|\, Q\in \So\}.\]
The Toda vector field on $\sl$ is tangent to $\cO$, and its equilibria there   are diagonal matrices. A result of Symes \cite{Sy,Sy2} provides the almost explicit calculation of the trajectories of the Toda vector field everywhere on $\sl$. Specifically, the trajectory starting at a symmetric matrix $X$ is given by
 \begin{equation}\label{eq:Toda-solution-Iwasawa-sl}
 \kappa(\mathrm{e}^{tX})^T  X  {\kappa(\mathrm{e}^{tX})},\quad t\in\R.
 \end{equation}

We define the map
\begin{equation}\label{eq:fat-chart}\Psi:\L\longrightarrow\cO,\quad L\mapsto \kappa(L)^T\, \Lambda\, \kappa(L).
 \end{equation}
The affine space $\L$ is canonically isomorphic to the vector space of strictly lower triangular matrices: Non-diagonal entries are linear variables (more conceptually,  one chooses the identity matrix as distinguished point/origin in the affine space $\L$). With this in mind, the vector field
\begin{equation}\label{eq:diagonal-linear}L'=[L,-\Lambda],\quad L\in \L,
\end{equation}
is a  diagonal linear vector field. Its trajectory starting at $L$ is given by conjugation by $\mathrm{e}^{t\Lambda}$.

The following proposition is a  known consequence of Symes' result.
\begin{proposition}\label{pro:Toda}
The map $\Psi: \L\to \cO $ takes  the trajectories of the action of $\mathrm{e}^{t\Lambda}$ on $\L$ by conjugation to the trajectories of the Toda vector field,
\[\Psi(\mathrm{e}^{t\Lambda}L\mathrm{e}^{-t\Lambda})=\kappa(\mathrm{e}^{tX})^T X  \kappa(\mathrm{e}^{tX}),
\quad X=\kappa(L)^T\Lambda \kappa(L).\]
Equivalently, the vector field  $L'=[L,-\Lambda]$  on $\L$
can be pushed forward by $\Psi$ and its image is the Toda vector field.
\end{proposition}

For a traceless diagonal matrix $\Lambda$ of simple spectrum,   it is known that $\Psi$ defined in  \eqref{eq:fat-chart} is a diffeomorphism onto its image, a dense open subset of $\cO$. Proposition \ref{pro:Toda} in this setting says that $\Psi$ is a change of coordinates that linearizes the Toda flow. We sketch the proof to highlight the role of factorization results. 

The map $\Psi$ is the composition
\[\L\overset{\kappa}{\longrightarrow } \So\overset{f}{\longrightarrow} \cO,\quad f(Q)=Q^T\Lambda Q,\]
of the first projection of the $\mathrm{QR}$ factorization in $\L$ with the anti-conjugation map. As observed, the first map is a diffeomorphism onto its image $\cC\subset \So$. It suffices to show that $f:\cC\to\cO$ is a diffeomorphism onto its image
\[\cU=\{Q^T\Lambda Q\,|\, Q\in \cC\}\subset \cO.\]
This can be done explicitly. However, and for the sake of the generalization to the non-simple spectrum case, we do it by looking at how the anti-conjugation map relates to a suitable factorization in $\So$. More precisely,  the preimage of $\cU$ by $f:\So\to \cO$ is the open and dense subset of $\So$ whose elements are matrices with nonzero minors, which factors as $\E\cC$, where $\E$ is the group of determinant one diagonal sign matrices. Because $\E$ are exactly the matrices in $\So$ that commute with $\Lambda$ --- the centralizer of $\Lambda$ in $\So$ --- it follows that the restriction
\[f:\E\cC\longrightarrow\cU\]
is a (trivial) fibration whose (finite) fibers are the right translates $\E Q$, $Q\in \cC$, of $\E$. This has two consequences. On the one hand $f:\cC\to\cU$ is a diffeomorphism because $\cC\subset \E\cC$ is a section to the fibration. On the other hand
\[\cU=f(\cC)=f(\E\cC)\] implies that $\cU$ is an open and dense subset of $\cO$.

In the case of multiple spectrum the map $\Psi$ ceases to be injective from counting dimensions. More precisely, the centralizers of $\Lambda$ in $\So$ and $\L$, that we denote by $\So_\Lambda$ and $\L_\Lambda$, respectively,
\begin{equation}\label{eq:lower-triangular-centralizer}
\So_\Lambda=\{Q\in \So\,|\, Q_{ij}=0,\,\,\mathrm{if}\,\,\lambda_i\neq \lambda_j\},\quad \L_\Lambda=\{L\in \L\,|\, L_{ij}=0,\,\,\mathrm{if}\,\, \lambda_i\neq \lambda_j\},
\end{equation}
are larger than $\E$ and the identity, respectively. The factors of the  $\mathrm{QR}$ factorization of a matrix that commutes with $\Lambda$ also commute with $\Lambda$. This implies that   $\kappa^{-1}(\So_\Lambda)\cap \L=\L_\Lambda$, from which  it follows
that
\[\Psi^{-1}(\Lambda)=\L_\Lambda,\]
and therefore $\Psi$ fails to be injective around the identity.

To look for coordinates for $\cO$ taking the diagonal linear vector field \eqref{eq:diagonal-linear} to the Toda vector field, it is natural to look for a submanifold of $\L$ which is tangent to \eqref{eq:diagonal-linear}, and to which $\Psi$ restricts to a diffeomorphism onto its image. The kernel of the differential of $\Psi$ at the identity is the tangent space to $\L_\Lambda$,
\begin{equation*}
\{L\in \L\,|\, L_{ij}=0,\,\,\mathrm{if}\,\, \lambda_i\neq \lambda_j\,\,\mathrm{or}\,\,i=j\}.
\end{equation*}
The affine subspace introduced in \eqref{eq:chart-domain}
\[\L^\Lambda=\{L\in \L\,|\, L_{ij}=0,\,\,\mathrm{if}\,\, \lambda_i=\lambda_j,\,i\neq j\}\]
is complementary to the kernel. Therefore the restriction
\[\Psi:\L^\Lambda\longrightarrow\cO\]
is a diffeomorphism around the identity.

To go from local to global diffeomorphism we use  Theorem \ref{thm:constrained-QLU}: The factorization
\[\So_\Lambda \L^\Lambda  \U\subset \Sl\]
holds in the complement in $\Sl$ of the solution set of a polynomial equation.
%
%

\begin{proof}[Proof of Theorem \ref{thm:main}]
Suppose that the constrained $\mathrm{Q}\L\U$ factorization holds. The first projection of the $\mathrm{QR}$ factorization  $\kappa:\Sl\to  \So$ to $\So_\Lambda\L^\Lambda\U$ yields
 \[\kappa(\So_\Lambda\L^\Lambda\U)=\So_\Lambda \kappa(\L^\Lambda),\]
a factorization of $\So_\Lambda$ defined  in  the complement of the solution set of a polynomial equation (the restriction to $\So_\Lambda$ of homogeneous polynomial coming from the constrained $\mathrm{Q}\L\U$ factorizaton). Notice that   $\So_\Lambda \kappa(\L^\Lambda)$ is different from  $\cC=\kappa(\L)$.

As
\begin{itemize}\item $\kappa:\L^\Lambda\to\kappa(\L^\Lambda)$ is a diffeomorphism,   \item $\Psi$ is the result of applying anti-conjugation to $\Lambda$ by elements in $\kappa(\L^\Lambda)$,\item   $\So_\Lambda$ is the centralizer of $\Lambda$ in $\So$,
 \item
 the (unique and smooth) factorization $\So_\Lambda \kappa(\L^\Lambda)\subset \So$ holds,
 \end{itemize}
 the same argument that we sketched for the simple spectrum case implies
that
\[\Psi:\L^\Lambda\longrightarrow\cO\]
is a diffeomorphism onto its image $\mathcal{V}\subset \cO$. Because $\So_\Lambda\kappa(\L^\Lambda)\subset \So$ is open and dense, then so is $\mathcal{V}\subset \cO$.

As $\L^\Lambda$ is a coordinate affine subspace of $\L$, the diagonal linear vector field  \eqref{eq:diagonal-linear} is tangent to
        $\L^\Lambda$. Therefore
Proposition \ref{pro:Toda} implies that in the local coordinates $\Psi^{-1}:\mathcal{V}\to \L^\Lambda$  the Toda vector field equals \eqref{eq:diagonal-linear}.


We argue that
$\cO$ is also covered by the union of the domains of the local coordinates centered at each diagonal matrix by using a result from dynamical systems: If a vector field on compact manifold has hyperbolic zeroes and has a strict Lyapunov function, then the union of its stable (or unstable) manifolds over its zeroes covers the manifold. Back to our setting,
\begin{itemize}
 \item the conjugacy class $\cO$ is compact,
 \item the Toda vector field on $\cO$ has hyperbolic zeroes, the diagonal matrices in the conjugacy class,
 \item  in the linear coordinates centered at a diagonal matrix  the stable manifold is a vector subspace and therefore it sits in the domain of the coordinates,
 \item the restriction to $\cO$ of the linear function
 \[f(X)=\mathrm{tr}(AX)\]
 where a is a fixed diagonal matrix with different strictly positive entries ordered increasingly, is a strict Lyapunov function for the Toda vector field  \cite[Lemma 1]{Fa}.
\end{itemize}
\end{proof}

 \begin{example}\label{ex:RP2} Let $\Lambda$ be the diagonal $3\times 3$ matrix with ordered entries $\{\lambda,-2\lambda,\lambda\}$, where $\lambda\in \mathbb{R}$ is nonzero. Its orthogonal conjugacy class
 is diffeomorphic to the projective plane $\mathbb{RP}^2$. It contains the  diagonal matrices
  with ordered entries $\{-2\lambda,\lambda,\lambda\}$,  $\Lambda$ and $\{\lambda,\lambda,-2\lambda\}$, which are a sink, a saddle and a source for the Toda vector field, respectively.
\begin{center}
\begin{figure}[H]
\includegraphics[scale=0.9]{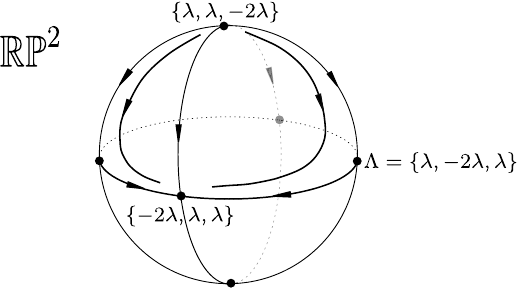}
 \caption{The Toda vector field on  $\mathbb{RP}^2$.}
\end{figure}
\end{center}
We check that the chart $\Psi$ centered around $\Lambda$ sends the diagonal vector field to the Toda vector field.
The domain $\L^\Lambda$ of $\Psi$ is a plane
\[\L^\Lambda=\left\{L=\begin{pmatrix} 1 & 0 & 0\\ x & 1 & 0\\ 0 & y & 1\end{pmatrix},\,\,x,y\in \R\right\},\]
and the diagonal linear vector field there  $L'=[L,-\Lambda]$ equals
\[\begin{pmatrix} 0 & 0 & 0\\ x' & 0 & 0\\ 0 & y' & 0\end{pmatrix}= 3 \lambda \begin{pmatrix} 0 & 0 & 0\\ - x & 0 & 0\\ 0 &  y & 0\end{pmatrix}.\]
The orthogonal factor of $L$ in the $\mathrm{QR}$ factorization equals
\[\kappa(L)=\begin{pmatrix} \dfrac{1}{n_1} & -\dfrac{x}{n_1n_2} & \dfrac{xy}{n_2}\\ \dfrac{x}{n_1} & \dfrac{1}{n_1n_2} & -\dfrac{y}{n_2}\\ 0 & \dfrac{n_1y}{n_2} & \dfrac{1}{n_2}\end{pmatrix},\quad n_1=\sqrt{1+x^2},\,n_2=\sqrt{1+y^2+x^2y^2},\]
and the conjugation of $\Lambda$ by $\kappa(L)^T$ is
\[\Psi(L)= \lambda \begin{pmatrix} \dfrac{(1-2x^2)}{n_1^2} & -\dfrac{3 x}{n_1^2n_2} & \dfrac{3 xy}{n_1n_2}\\ -\dfrac{3 x}{n_1^2n_2} & \dfrac{(-2+x^2+y^2(1+x^2)^2)}{n_1^2n_2^2} & \dfrac{3 y}{n_1n_2^2}\\ \dfrac{3 xy}{n_1n_2} & \dfrac{3 y}{n_1n_2^2} & \dfrac{(1-2y^2+x^2y^2)}{n_2^2}\end{pmatrix}\]
The Toda vector field  $[\Psi(L),\pi_{\so}\Psi(L)]$  equals
	\[9 \lambda^2 \begin{pmatrix} \frac{2 x^2}{n_1^4} &-\frac{ x(-1+x^2-(2+x^2-x^4)y^2)}{n_1^4n_2^{3}}  & \frac{xy(x^2+(-1+x^4)y^2)}{n_1^{3}n_2^{3}} \\ 
		- \frac{ x(-1+x^2-(2+x^2-x^4)y^2)}{n_1^4n_2^{3}} & -\frac{2(x^2-(1-x^4)y^2)}{n_1^4n_2^4} &\frac{(1-y^2+x^2(2+y^2+2x^2y^2)) y}{n_1^3n_2^4} \\ 
		\frac{xy(x^2+(-1+x^4)y^2)}{n_1^{3}n_2^{3}} & \frac{(1-y^2+x^2(2+y^2+2x^2y^2)) y}{n_1^3n_2^4} & -\frac{2(1+x^2y^2)y^2}{n_2^4}\end{pmatrix} . \]
		The reader may verify that, indeed,
		\[ D \Psi(L) ([L, - \Lambda ])= 3 \lambda  \left( \frac{ \partial \Psi(L)}{\partial x} \cdot x  - \frac{\partial \Psi(L)}{\partial y} \cdot y \right) \  = [\Psi(L),\pi_{\so}\Psi(L)] \, .\]
\end{example}

\begin{remark} Let $\cO$ be the orthogonal conjugacy class of the diagonal matrix $\Lambda$ with ordered entries $\lambda_1, \dots, \lambda_1, \lambda_{2}, \dots, \lambda_2$, $\lambda_1 < \lambda_2$, where $\lambda_1$ and $\lambda_2$ appear respectively $d$ and $n-d$ times. There is a canonical diffeomorphism between $\cO$ and the Grassmannian $\mathrm{Gr}(d,n)$: Take the eigenspace $F \subset \RR^n$ of the matrix $X \in \cO$ associated with $\lambda_1$. Under this diffeomorphism, the equilibria of the Toda vector field correspond to coordinate subspaces of dimension $d$. 
For the previous example in 	$\mathbb{RP}^2$, the equilibria correspond to the three coordinate axes.

Theorem \ref{thm:main} produces a new atlas for the Grassmanian which linearizes the Toda flow. In $\mathbb{RP}^2$, the charts centered at the previous diagonal matrices are

%
%
%

\begin{flalign*}
 \begin{pmatrix} 1 & 0 & 0\\ x & 1 & 0\\ y & 0 & 1\end{pmatrix}  \equiv (x,y) \longmapsto  &\begin{bmatrix} 1 & - \dfrac{x}{\sqrt{1 + y^2 }} & - \dfrac{y \sqrt{1 + x^2 + y^2 }}{\sqrt{1 + y^2 }} \end{bmatrix}, &\\
 \begin{pmatrix} 1 & 0 & 0\\ x & 1 & 0\\ 0 & y & 1\end{pmatrix}  \equiv (x,y) \longmapsto  &\begin{bmatrix} x \sqrt{1 + y^2 + x^2 y^2} & 1 & - y \sqrt{1 + x^2} \end{bmatrix}, & \\
  \begin{pmatrix} 1 & 0 & 0\\ 0 & 1 & 0\\ x & y & 1\end{pmatrix} \equiv (x,y) \longmapsto  &\begin{bmatrix}  \dfrac{x \sqrt{1 + x^2 + y^2 }}{\sqrt{1 + x^2 }} &  \dfrac{y}{\sqrt{1 + x^2 }} & 1 \end{bmatrix}. &
\end{flalign*}

Notice that these charts are obtained by a nonlinear rescaling on the coordinates of the standard charts of $\mathbb{RP}^2$.

\end{remark}

\section{The constrained $\mathrm{QLU}$ factorization}\label{sec:constrained-QLU}
 All matrices in this section all real.
We discuss the constrained $\mathrm{Q}\L\U$ factorization in the general linear group $\Gl$, rather than for the special linear one. We show how to obtain, away from the solution set of a polynomial equation,  a unique and smooth factorization of an invertible matrix into a orthogonal factor, a lower triangular one, and an upper triangular one with positive diagonal entries (we abuse notation and still denote the subgroup by $\U$),
\[X=QLU,\quad Q\in \mathrm{O},\,L\in \L,\,U\in \U,\] so that $Q$ and $L$ have zero entries in positions that obey some appropriate patterns.

To any given  diagonal matrix $\Lambda$ of size $n$ we associate the following data.
\begin{itemize}
 \item
An ordered partition of the  set
$\{1,\dots,n\}$ into subsets $\Lambda_{1},\dots,\Lambda_{k}$ that collect the positions of the equal eigenvalues of $\Lambda$ (the positions on each subset need not be consecutive). We assume  without loss of generality that if  $i<j$, then the first position in $\Lambda_{i}$ is smaller than the first position in $\Lambda_{j}$. For  $i\in \{1,\dots,n\}$, let $\Lambda_{(i)}$ the subset that contains $i$.
 \item The direct sum decomposition  into coordinate subspaces
\[\R^n=V_{1}\oplus\cdots \oplus V_{k},\]
where $V_{i}$ is spanned by the coordinates with indices in $\Lambda_{i}$.
\item The subgroup $\mathrm{O}_\Lambda\subset \Gl$ of orthogonal matrices that commute with $\Lambda$, or, equivalently, the matrices that leave invariant all $V_i$.
\item The space $\L^\Lambda$  introduced in \eqref{eq:chart-domain}:  unit triangular matrices whose subdiagonal entry  $(ij)$ is zero if $i$ and $j$ belong to the same $\Lambda_{s}$.
\end{itemize}

%

Consider the product 
\[\mathrm{O}_\Lambda\L^\Lambda \U=\{X\in \Gl\,|\ X=QLU,\quad Q\in\mathrm{O}_\Lambda,\,L\in \L^\Lambda,\,U\in \U\}.\]
An application of the inverse function theorem shows that there is a neighborhood of the identity matrix  where the factorization is unique.

\begin{theorem}\label{thm:constrained-QLU-detailed} The constrained $\mathrm{Q}\L\U$ factorization for invertible matrices
 \[\mathrm{O}_\Lambda\L^\Lambda \U\subset \Gl\]
has the following properties.
\begin{enumerate}
 \item (Uniqueness and smoothness) The factorization is unique and it depends smoothly on parameters.
 \item (Density) The factorization exists for matrices  in
 \[\mathfrak{C}^{\Lambda}=\{X\in \Gl\ |\  p_1\cdot p_2\cdots p_{n-1}(X)\neq 0\},\]
 where each $p_i$ is a homogeneous polynomial in the entries of $X$. In particular, $\mathfrak{C}^{\Lambda}$ is an open dense subset of $\Gl$.
 \item (Pivoting) For any $X\in \Gl$ a permutation $P$ exists such that $PX\in \mathfrak{C}^{P\Lambda P^T}$. Equivalently,
 \[\Gl=\bigcup_{P\in \cS_n}P\, \mathfrak{C}^{P^T\Lambda P},\]
 where $\cS_n$ is the group of permutation matrices.
 \end{enumerate}
\end{theorem}

\begin{example}
Two examples  of the constrained $\mathrm{Q}\L\U$ factorization are familiar.
 \begin{enumerate}
  \item $\Lambda_1=\{1,\dots,n\}\Longrightarrow \Gl_\Lambda=\Gl,\,\L^\Lambda=\{\mathrm{I}\}$.\newline
  In this case we recover the $\mathrm{QR}$ factorization, that is unobstructed.
  \item $\Lambda_i=\{i\}\Longrightarrow \Gl_\Lambda=\E,\,\L^\Lambda=\L$.
  In this case we recover the $\L\U$ factorization, up to  multiplication on the left by a sign matrix. The obstruction is that some principal minor of $X$ be zero.
 \end{enumerate}

 As the subsets of repeated eigenvalues of $\Lambda$ range through all possible partitions of the ordered set $\{1,2,\dots,n\}$, the corresponding instances of the constrained $\mathrm{Q}\L\U$ factorization provide an interpolation of sorts between the  $\mathrm{QR}$ and (signed) $\L\U$ factorizations.
\end{example}
\begin{proof}
Given $X\in \Gl$ our purpose is
\begin{enumerate}[(i)]
 \item to
describe an algorithm that finds inductively the $i+1$-th column of $Q\in \mathrm{O}_\Lambda$ (and the $i+1$-th column of $L\in \L^\Lambda$) such that
\[X=QLU,\quad U\in \U;\]
\item to state the obstruction to the $i+1$-th step as the vanishing of certain homogeneous polynomial in the entries of the first $i$ columns of $X$.
\end{enumerate}
Our algorithm for the constrained $\mathrm{QLU}$ factorization of a given matrix $X\in \Gl$ stems from thinking that we give ourselves the freedom to multiply $X$ by the left by $\mathrm{O}_\Lambda$, so that the $\mathrm{LU}$ factorization of the ensuing product has unit lower triangular factor in $\L^\Lambda$.

More precisely, let $c_1,\dots,c_n$ denote the columns of the given matrix $X\in \Gl$, and let $r_1,\dots,r_n$ denote the rows of the unknown $Q^T\in \mathrm{O}_\Lambda$, so that $Q^T X$  posseses an $\L\U$ factorization with first factor in $\L^\Lambda$,
\[Q^T X=LU.\]
 We solve the equivalent equation
\[Q^T XY=L, \quad L\in \L^\Lambda,\quad Y=U^{-1}.\]

Let $\pr_{i}:\R^n\to V_i$ denote the orthogonal projection. Consider the homogeneous degree two polynomials in the entries of $X$ defined by  the inner product
\[y_{i,jk}=\pr_i(c_j)\cdot \pr_i(c_k).\]
\vskip .2cm

\noindent{\underline{Step 1:}} The entries of first column of the matrix $Q^T X$ are
 $r_j\cdot c_1$. The
 equality
 \[Q^T XY=L\]
 implies that the first column of $L$ is the  normalization of the column $r_j\cdot c_1$ by $Y_{11}:=r_1\cdot c_1$.
    For $L$ to be in $\L^\Lambda$,
     the condition on its first column is
  that its first entry is nonzero and any other entry in $\Lambda_1$ is zero,
 \begin{equation}\label{eq:orthogonal}r_1\cdot \pr_{1}(c_1)\neq 0,\quad r_{l_2}\cdot \pr_{1}(c_1)=0,\,\dots\,, r_{l_a}\cdot \pr_{1}(c_1)=0,\quad \Lambda_1=\{1,l_2,\dots,l_a\},
 \end{equation}
 where we have written projections for the column vector $c_1$ because  $r_1,r_{l_i}\in V_1$. The first equation  holds if and only if
 \[R_1:=\pr_{1}(c_1)\neq 0\Longleftrightarrow y_{1,11}\neq 0\quad (y_{1,11}=\pr_{1}(c_1)\cdot \pr_{1}(c_1)).\]
 We assume this is the case.
 Because $r_1,r_{l_2},\dots,r_{l_a}$ must be an orthogonal basis of $V_1$, the equations in \eqref{eq:orthogonal} from the second to the last one imply that
 \[r_1=\frac{\pr_{1}(c_1)}{|\pr_{1}(c_1)|}.\]
 There is just one choice for $r_1$ and $Y_{11}=r_1\cdot c_1>0$.
\vskip .2cm
 \noindent{\underline{Step 2:}}  We look for $r_2$ (and for $Y_{12}$ and $Y_{22}$) so that the linear system
 \[\begin{pmatrix} r_1\cdot c_1 & r_1\cdot c_2\\
  r_2\cdot c_1 & r_2\cdot c_2\\
   \end{pmatrix}\begin{pmatrix} Y_{12}\\Y_{22}\end{pmatrix}
   =\begin{pmatrix} 0\\1\end{pmatrix},\quad Y_{22}>0\]
     has solution and the ensuing second column of $L$ is such that $L\in \L^\Lambda$. The system can be rewritten in the more geometric form
      \[\begin{pmatrix} r_1\cdot (Y_{12}c_1+ Y_{22}c_2)\\
  r_2\cdot (Y_{12}c_1+ Y_{22}c_2)
   \end{pmatrix}
   =\begin{pmatrix} 0\\1\end{pmatrix}\Longleftrightarrow \begin{pmatrix} R_1\cdot (Y_{12}c_1+ Y_{22}c_2)\\
  r_2\cdot (Y_{12}c_1+Y_{22}c_2)\end{pmatrix}
   =\begin{pmatrix} 0\\1\end{pmatrix}.\]
  \begin{itemize}
\item \emph{Case 1}: $\Lambda_{(2)}=\Lambda_1$. The system becomes
\[
   \begin{pmatrix} R_1\cdot (Y_{12}\pr_1(c_1)+ Y_{22}\pr_1(c_2))\\
  r_2\cdot (Y_{12}\pr_1(c_1)+Y_{22}\pr_1(c_2))\end{pmatrix}
   =\begin{pmatrix} 0\\1\end{pmatrix}.\]
   Because $r_1\cdot c_1\neq 0$, there is a 1-parameter family of solutions of the first equation, one of which is
  \[R_2=-y_{1,12}\pr_{1}(c_1)+ y_{1,11}\pr_{1}(c_2),\quad (y_{1,11}>0).\]
  We  assume that this is not the trivial vector.
The first equation, together with the set of equations in  \eqref{eq:orthogonal} for $l_j$, $j>1$ (that come from  $L\in \L^\Lambda$), imply that the unique solution  $r_2$ is the normalization
\[r_2=\frac{R_2}{||R_2||},\quad (Y_{12}=-\frac{y_{1,12}}{||R_2||},\quad Y_{22}=\frac{y_{1,11}}{||R_2||}>0).\]
   \item \emph{Case 2}: $\Lambda_{(2)}\neq  \Lambda_1$. The system becomes
    \[\begin{pmatrix} R_1\cdot (Y_{12}\pr_{1}(c_1) +Y_{22}\pr_{1}(c_2))\\
  r_2\cdot (Y_{12}\pr_{2}(c_1)+Y_{22}\pr_{2}( c_2))\\
   \end{pmatrix}
   =\begin{pmatrix} 0\\1\end{pmatrix}.\]
   As in the previous case,
 \[-y_{1,12}\pr_{1}(c_1)+ y_{1,11}\pr_{1}(c_2),\quad (y_{1,11}>0)\]
 is a solution to the first equation (that has a 1-parameter family of solutions).
The second equation holds if and only if  
 \[R_2=-y_{1,12}\pr_{2}(c_1)+ y_{1,11}\pr_{2}(c_2)\neq 0.\]
We assume this is the case.
For the second column of $L$ to be in $\L^\Lambda$,
\[ r_{m_j}\cdot R_2=0,\quad \Lambda_{(2)}=\{ 2,m_2,\dots, m_b\},\]
which imply that the unique solution is
\[r_2=\frac{R_2}{||R_2||}\quad (Y_{12}=-\frac{y_{1,12}}{||R_2||},\quad Y_{22}=\frac{y_{1,11}}{||R_2||}>0).\]
Note that because $\Lambda_1\neq \Lambda_{(2)}$, the rows $r_1,r_2$ are  orthogonal, and that the coefficients of $R_2$ are homogeneous degree one polynomials in $y_{i,jk}$.
\end{itemize}
\vskip .2cm
We finish \underline{Step 2} observing that the equality
 \[\begin{pmatrix} r_1\cdot c_1 & r_1\cdot c_2\\
  r_2\cdot c_1 & r_2\cdot c_2\\
   \end{pmatrix}\begin{pmatrix}Y_{11} & Y_{12}\\ 0 & Y_{22}\end{pmatrix}
   =\begin{pmatrix} 1 & 0 \\ L_{21} & 1\end{pmatrix}\]
implies
\[\begin{vmatrix} r_1\cdot c_1 & r_1\cdot c_2\\
  r_2\cdot c_1 & r_2\cdot c_2\\
   \end{vmatrix}\neq 0 \Longleftrightarrow p_2:=\begin{vmatrix} R_1\cdot c_1 & R_1\cdot c_2\\
   R_2\cdot c_1 & R_2\cdot c_2
  \end{vmatrix}\neq 0.
  \]
 In {\emph{Case 2}} the determinant $p_2$ is
  \[\begin{vmatrix} y_{1,11} & y_{1,12}\\
  -y_{1,12}y_{2,11}+ y_{1,11}y_{2,12} & -y_{1,12}y_{2,12}+ y_{1,11}y_{2,22}
  \end{vmatrix}.\]
  In {\emph{Case 1}}  the expression is analogous. In the second row, $y_{2,ij}$ is replaced by $y_{1,ij}$,
  \[p_2=y_{1,11}(-y_{1,12}y_{1,12}+ y_{1,11}y_{1,22}).\]
  Therefore, the assumptions to apply \underline{Step 1} and \underline{Step 2} are equivalent to nonvanishing conditions on determinants
  \[R_1\neq 0,\, R_2\neq 0\Longleftrightarrow p_1:=y_{1,11}\neq 0,\, p_2\neq 0.\]
\vskip .2cm

\noindent{\underline{Step i+1:}} We look for $r_{i+1}$ (and $Y_{1,i+1},\dots Y_{i+1,i+1}$) a solution
of the linear system
\[\begin{pmatrix} r_1\cdot c_1 &  \ldots & r_1\cdot c_i &  r_1\cdot c_{i+1} \\
\vdots  & \ddots & \vdots &\vdots \\
r_i\cdot c_1 & \ldots & r_i\cdot c_i & r_i\cdot c_{i+1}\\
r_{i+1}\cdot c_1 & \ldots & r_{i+1}\cdot c_i & r_{i+1}\cdot c_{i+1}
  \end{pmatrix}\begin{pmatrix} Y_{1,i+1} \\\vdots \\ Y_{i,i+1}\\ Y_{i+1,i+1}\end{pmatrix}=\begin{pmatrix} 0 \\\vdots \\ 0\\1\end{pmatrix},\quad Y_{i+1,i+1}>0,
   \]
  so that the ensuing column of $L$ is such that $L\in \L^\Lambda$.
By induction
\[p_i=\begin{vmatrix} R_1\cdot c_1 &  \ldots & R_1\cdot c_i \\
\vdots  & \ddots & \vdots \\
R_i\cdot c_1 & \ldots & R_i\cdot c_i & \end{vmatrix}\neq 0,\]
so the subsystem of the first $i$ equations has a 1-parameter family of solutions. Among them, one is of the form
\[(x_1,\dots,c_i,|p_i|),\]
that we assume to be nonzero.
The  equations for indices in $\Lambda_{(i+1)}$ different from $i+1$, accounting for the vanishing of the corresponding rows in the $i+1$-th column of $L$,  imply that the unique solution $r_{i+1}$ is the normalization of
\[R_{i+1}=x_1\pr_{(i+1)}(c_1)+\cdots+x_i\pr_{(i+1)}(c_i)+|p_i|\pr_{(i+1)}(c_{i+1}),\]
which we assume to be nontrivial.
 The equations in the subsystem for indices in $\Lambda_{(i+1)}$ imply that $r_{i+1}$ is orthogonal to those rows. (It is automatically orthogonal to rows with indices not in $\Lambda_{(i+1)}$).

The vector of coefficients $(x_1,\dots,x_i)$ of $R_{i+1}$ is the solution of
\[\begin{pmatrix} R_1\cdot c_1 &  \ldots & R_1\cdot c_i \\
\vdots  & \ddots & \vdots  \\
R_i\cdot c_1 & \ldots & R_i\cdot c_i
  \end{pmatrix}\begin{pmatrix} x_1 \\\vdots \\ x_i\end{pmatrix}=\begin{pmatrix} -|p_i|R_1\cdot \pr_{1}(c_{i+1}) \\\vdots \\ -|p_i|R_i\cdot \pr_{(i)}(c_{i+1})\end{pmatrix}.
   \]
 Each row of the extended matrix of the system is, by induction, a homogeneous polynomial in $y_{i,jk}$. By Cramer's rule $(x_1,\dots,x_i)$ are homogeneous polynomials in $y_{i,jk}$ (of the same degree). Hence the coefficients of
 \[R_{i+1}=x_1\pr_{(i+1)}(c_1)+\cdots+\pr_{(i+1)}(c_i)+|p_i|\pr_{(i+1)}(c_{i+1})\]
are also homogeneous polynomials in $y_{i,jk}$ (of the same degree). Thus $p_{i+1}$ is a homogeneous polynomial in these variables.

The conclusion is that if we choose $X$ such that
\[p_1\cdots p_{n-1}(X)\neq 0,\]
then $X$ has a unique constrained $\mathrm{Q}\L\U$ factorization (the last step/polynomial is unobstructed because $X$ is invertible).

The uniqueness of the factorization follows from the uniqueness of each step. In particular, given $X=QLU$, then because 
\[Q^T XY=L,\quad Y=U^{-1},\]
one checks that the algorithm must be unobstructed, this implying that  $\mathrm{O}_\Lambda L^\Lambda\mathrm{U}$ is the complement of the algebraic variety $p_1\cdots p_{n-1}=0$. The smooth dependence on parameters of the algorithm is clear (in fact, it is real analytic), and the three factors are submanifolds. Therefore properties (1) and (2) hold.

As for the pivoting condition, we remark that in the case $\Lambda_{i}=\{i\}$, that returns the $\L\U$ factorization, for any permutation $P$ the partitions associated to   $P\Lambda P^T$ and $\Lambda$ agree. Therefore we recover the classical pivoting condition.

For an arbitrary $\Lambda$, rather than making a finer analysis of the  zero set of the polynomial
$p_1\cdots p_{n-1}$ for all permutations, we appeal to Theorem \ref{thm:main}. That the union of the domains of the local coordinates covers $\cO$ translates into
\[\So=\bigcup_{P\in \cS_n} P\, \So(P^T\Lambda P)\, \kappa(\L^{P^T\Lambda P}),\]
in the notation of Theorem \ref{thm:main}, which is equivalent to
\[\Gl=\bigcup_{P\in \cS_n} P\, \mathfrak{C}^{P^T\Lambda P}.\]
There is redundancy among the subsets above: Two permutations $P_1,P_2$ yield the same open dense subset of $\Gl$ if and only if $P^T_1P_2\in \mathrm{O}_\Lambda$.
\end{proof}

\subsection{Another algorithm for the constrained $\mathrm{Q}\L\U$ factorization}
Theorem \ref{thm:constrained-QLU-detailed} describes an algorithm for the constrained $\mathrm{Q}\L\U$ factorization based on the $\L\U$ factorization: Given $X\in \Gl$, the main unknown is  $Q\in \mathrm{O}_\Lambda$ subject to
\[Q^TX=LU,\quad L\in \L^\Lambda.\]

We present a second algorithm where the main unknown is $U\in \U$ subject to
\begin{equation}\label{eq:algorithm-2}
XU^{-1}\in \mathrm{O}_\Lambda \L^\Lambda.
 \end{equation}
This approach leans on the explicit description of $\mathrm{O}_\Lambda \L^\Lambda$. For instance, for $n=4$ and $\Lambda_1=\{1,3\}$, $\Lambda_2=\{2,4\}$, the product of
					\[ Q \ = \ \begin{pmatrix} Q_{11}& 0 & Q_{13}& 0\\ 0& Q_{22}& 0& Q_{24}\\ Q_{31}& 0& Q_{33}& 0\\ 0& Q_{42}& 0& Q_{44} \end{pmatrix}, \quad  L \ = \  \begin{pmatrix} 1& 0& 0& 0\\ L_{21}& 1& 0& 0\\ 0& L_{32}& 1& 0\\ L_{41}& 0& L_{43} & 1\end{pmatrix}, \]equals  \[QL=\begin{pmatrix}
						Q_{11} & Q_{13} L_{32}& Q_{13} & 0\\
						Q_{22} L_{21}+Q_{24} L_{41}  & Q_{22}  & Q_{24} L_{43}  & Q_{24}\\
						Q_{3 1}  & Q_{3 3}L_{32}  & Q_{3 3}  & 0\\
						Q_{4 2} L_{21}+Q_{4 4} L_{41}  & Q_{4 2}  & Q_{4 4} L_{43}  & Q_{4 4} \end{pmatrix}  \]
From this example one deduces that in general
\[QL=Q(\mathrm{I}+(L-\mathrm{I}))=Q+N,\quad Q\in\mathrm{O}_\Lambda,\,L\in \L^\Lambda,\]
where $N=Q(L-\mathrm{I})$ is characterized as follows: Its $i$-th column $c_i(N)$
satisfies
\[\pr_{(i)}(c_i(N))=0,\qquad \pr_{j}(c_i)=\sum_{k>i,k\in \Lambda_j}L_{ki}\pr_{j}(c_k(N)),\quad j\neq (i).\]
In other words,
\begin{itemize}
 \item  The projection of $N$ onto $V_i$ is trivial.
 \item The projection onto $V_j\neq V_i$  is a linear combination of the orthonormal subframe $\pr_j(c_k(Q))$, $k>i,k\in \Lambda_j$, of $V_j$. Equivalently, it is perpendicular to  $\pr_j(c_k(Q))$, $k<i,k\in \Lambda_j$.
\end{itemize}

Likewise, the computation of the product $QLU$ in the example above suggests a way to solve \eqref{eq:algorithm-2} by writing $U^{-1}=U^{[1]}\cdots U^{[n]}$, where the only nonzero off diagonal entries of $U^{[i]}$ occur in the $i$-th row.

To carry out Step 1 we assume that   $\pr_1(c_1(X))\neq 0$. Then $U^{[1]}$ is uniquely determined by
\[U^{[1]}_{11}=||\pr_1(c_1(X))||,\quad (U^{[1]}_{1k}\pr_1(c_1(X))+\pr_1(c_k(X))\cdot \pr_1(c_1(X))=0.\]
The first column of  $X^{[1]}=XU^{[1]}$ satisfies the requirements to belong to $\mathrm{O}_\Lambda \L^\Lambda$, and the projection of the remaining columns onto $V_1$ are perpendicular to $\pr_1(c_1(X))$.

We assume that we have carried out Step $i$ and obtained $X^{[i]}=XU^{[1]}\cdots U^{[i]}$ whose first $i$ columns satisfy the requirements to belong to $\mathrm{O}_\Lambda\L^\Lambda$, and whose remaining columns are perpendicular to $\pr_{1}(c_1(X^{[i]}),\dots,\pr_{(i)}(c_i(X^{[i]})$. To carry out Step $i+1$ we assume that
 $\pr_1(c_{i+1}(X^{[i]})\neq 0$ and compute $U^{[i+1]}$ by solving
 \[U^{[i+1]}_{i+1\,i+1}=||\pr_{(i+1)}(c_{i+1}(X^{[i]}))||,\]
 \[(U^{[i+1]}_{i+1\,k}\pr_{(i+1)}(c_{i+1}(X^{[i]}))+\pr_{(i+1)}(c_k(X^{[i]}))\cdot \pr_{(i+1)}(c_{i+1}(X^{[i]}))=0.\]

By induction, we obtain $XU^{[1]}\cdots U^{[n]}\in \mathrm{O}_\Lambda\L^\Lambda$, thus yielding another proof of the unique and smoothly dependent factorization  $\mathrm{O}_\Lambda\L^\Lambda\U\subset \Gl$.

\subsection{Matrices with complex or quaternionic entries}

Linearization and constrained factorization  also hold for matrices with complex or quaternionic entries.

In the complex setting, we split traceless complex matrices into skew-Hermitian matrices and upper triangular matrices with real diagonal entries
\begin{equation}\label{eq:Toda-complex}\sl_\C=\mathfrak{k}\oplus \uu.
 \end{equation}
The Toda vector field is defined as in \eqref{eq:Toda} by
\[X'=[X,\pi_{\mathfrak{k}}X].\]
 The Hermitian conjugacy class $\cO\subset \sl_\C$ consist of Hermitian matrices with fixed spectrum. Equivalently, it is the result of conjugating a traceless real diagonal matrix by the special unitary group $\mathrm{K}$. The group $\L\subset \Sl_\C$ now consists of unit lower triangular matrices with complex coefficients. Theorem \ref{thm:main} holds in this setting: Around a real diagonal matrix $\Lambda\in \cO$ there are coordinates with values  in $\L^\Lambda\subset \L$ in which the Toda vector field reads
 \[L'=[L,-\Lambda],\quad L\in \L^\Lambda,\]
  where $\L^\Lambda$ is defined by the same equations \eqref{eq:chart-domain} as in the case of real matrices.

The proof is the same. Symes' result holds for the trajectories of the Toda vector field defined by an Iwasawa decomposition of a real semisimple Lie algebra, as in \eqref{eq:Toda-complex}, and
\[\Psi:\L\to \cO,\quad L\mapsto \kappa(L)^T\Lambda\kappa(L),\]
also sends $L'=[L,-\Lambda]$, $L\in \L$, to the Toda vector field, where $\kappa:\Sl_\mathbb{C}\to \mathrm{K}$ is the first projection associated to the $\mathrm{QR}$ factorization with respect to the standard Hermitian inner product in $\C^n$. Therefore the proof of the Theorem \ref{thm:main} rests again on the analog of Theorem \ref{thm:constrained-QLU} for complex matrices: The existence of a constrained $\mathrm{Q}\L\U$ factorization
\[\mathrm{K}_\Lambda \L^\Lambda \U\subset \Sl_\C, \quad \mathrm{K}_\Lambda=\{K\in \mathrm{K}\,|\, K\Lambda=\Lambda K\},\]
which is unique, depends on parameters in a smooth (real analytic) fashion, and is defined in the complement of the zero set of  homogeneous real valued polynomials on the real and imaginary parts of the entries of $X$. The counterpart to Theorem \ref{thm:constrained-QLU}  can be proved by applying either of the algorithms we described for real matrices, with minor adjustments. More precisely, for the first algorithm we solve
\[\overline{K}X=LU,\quad K\in \mathrm{K}_\Lambda,\,L\in \L^\Lambda,\]
so that entries of the matrix product   $\overline{K} X$ can be expressed using the standard Hermitian inner product. Because $\mathrm{K}$ is invariant under complex conjugation, it follows that $\overline{K}\in \mathrm{K}_\Lambda$, as desired.

As for quaternionic coefficients, we regard $\sl_\mathbb{H}(n)\subset \sl_\C(2n)$ as
matrices whose entries are $2\times 2$ complex matrices that correspond to quaternions, and whose trace is a quaternion with trivial scalar part. We recall that the $\mathrm{QR}$ factorization on $\Sl_\C$ applied to a quaternionic matrix returns quaternionic factors \cite{Sa}. One  consequence is that, upon differentiation, we obtain the direct sum decomposition
\[\sl_\mathbb{H}=(\mathfrak{k}\cap  \sl_\mathbb{H}) \oplus (\uu\oplus \sl_\mathbb{H}),\]
which implies that the Toda vector field on $\sl_\C$ is tangent to $\sl_\mathbb{H}$.
Another consequence is  that the restriction of $\Psi:\L\to \cO$ to $\L\cap \Sl_\mathbb{H}$ takes values in the quaternionic Hermitian conjugacy class $\cO\cap \Sl_\mathbb{H}$. Therefore Theorem \ref{thm:main} for quaternionic coefficients is a consequence of the existence of a  constrained $\mathrm{Q}\L\U$ factorization. The factorization in turn can be obtained by applying the first algorithm, where we  regard $X\in \Sl_\mathbb{H}$ as a matrix with quaternionic entries, and not as a matrix in $\sl_{\C}(2n)$, and we use the quaternionic standard inner product in $\mathbb{H}^n$ (understood as a vector space by right multiplication) and the invariance of unitary quaternionic matrices under quaternionic conjugation.

\subsubsection{Some Lie theoretic aspects of the constrained factorization}\label{ssec:Lie-theory}
When the repeated eigenvalues of $\Lambda$ are contiguous, the vector subspace  \[\ll^\Lambda=\{L-\mathrm{I}\,|\,L\in \L^\Lambda\}\] is also a Lie algebra. (These are equivalent statements). In such situation the constrained $\mathrm{Q}\L\U$ factorization is an immediate consequence of the direct sum of Lie algebras
$\sl=\so_\Lambda\oplus \ll^\Lambda\oplus \uu$,
because this implies that the product map
\[\So_\Lambda\times \L^\Lambda\times \U\longrightarrow\So_\Lambda \L^\Lambda\U\]
is a diffeomorphism onto its image \cite[Lemma 6.44]{Kn}. Such charts, however, are not sufficient to cover $\Gl$.

For $\Lambda$ with non-simple spectrum, the unique factorization $\L=\L_\Lambda\L^\Lambda$  and the normalization of $\L^\Lambda$ by $\L_\Lambda$ still hold. One can define a composition law
\[\L^\Lambda\times \L^\Lambda\longrightarrow \, \L^\Lambda,\quad L_1\ast L_2\, = \, \mathrm{pr}_2(L_1L_2),\quad \pr_2:\L\longrightarrow\L^\Lambda.\]
This operation is not associative but has left and right inverses ($(\L^\Lambda,*)$ is a so-called loop). More precisely, the equation
\[L_1L_2=LL_3, \quad L_i\in \L^\Lambda,\,L\in \L_\Lambda,\]
for given $L_1,L_3$ (resp. $L_2,L_3$) and unknowns $L_2,L$ (resp. $L_1,L$) has a unique solution depending analytically on parameters. We do not know whether this loop structure may be related to  the existence of the constrained $\mathrm{Q}\L\U$ factorization.

The Toda vector field is defined on an arbitrary real semisimple Lie algebra $\gg$ once an Iwasawa decomposition $\gg=\kk\oplus\aa\oplus \nn$ has been fixed (with underlying Cartan decomposition $\gg=\kk\oplus \ss$). The generalization of an orthogonal conjugacy class of symmetric matrices is an adjoint orbit of the fixed maximal compact subgroup of an element in $\ss$. Diagonal matrices become elements $\Lambda\in \aa$
and simple spectrum corresponds to the nonvanishing of the roots on $\Lambda$ (regularity).

In \cite{MT}
we constructed local coordinates diagonalizing the Toda vector field on regular symmetric adjoint orbits centered at elements $\Lambda$, that mapped onto  the sum of negative root spaces $\ll$. Note that the group $\L$ is no longer a vector space in general, so it cannot be the target of local coordinates. If $\Lambda$ is not regular, then $\ll$ decomposes into $\ll_\Lambda\oplus \ll^\Lambda$, where the first summand is the isotropy Lie algebra for the action of $\ll$ and the second is the sum of negative root spaces for roots not vanishing at $\Lambda$. The latter subspace is a subalgebra if  $\Lambda$ is in the positive Weyl chamber for the root decomposition. In such case, there is a constrained $\mathrm{Q}\L\U$ factorization  ($\U=\mathrm{A}\mathrm{N}$), and coordinates around $\Lambda$ on its symmetric adjoint orbit diagonalizing the Toda vector field exist.

To produce such coordinates around arbitrary $\Lambda$, a natural attempt from a Lie algebra perspective would be to replace  $\L^\Lambda$ by $\exp(\ll^\Lambda)$. However, as kindly pointed out to us by Qiyuan Gu, $\So_\Lambda\exp(\ll^\Lambda)\U$ already fails to be a factorization in $\Sl_\R(3)$ for $\Lambda=\{\lambda,-2\lambda,\lambda\}$. Indeed, the two products below are factorizations of this kind of the same matrix,
\[\begin{pmatrix} 0 & 0 & 1\\0 & -1 & 0\\ 1 & 0 & 0\end{pmatrix}\cdot\mathrm{exp}\begin{pmatrix} 0 & 0 & 0 \\ -\sqrt{2} & 0 & 0\\  0& -\sqrt{2}  &0 \end{pmatrix}=\mathrm{exp}\begin{pmatrix} 0 & 0 & 0 \\ \sqrt{2} & 0 & 0\\ 0& \sqrt{2}  &0 \end{pmatrix}\cdot\begin{pmatrix} 1 & -\sqrt{2} & 1 \\ 0 & 1 & -\sqrt{2}\\ 0 & 0 & 1 \end{pmatrix}.
\]

\subsubsection{Classical matrix groups}
A  strategy for a classical matrix group  $\mathrm{G}\subset \Sl_\C$  would be to adjust what we did for  $\Sl_\mathbb{H}\subset \Sl_\C$ along the following lines:

\begin{enumerate}[(i)]
 \item Conjugate the subgroup $\mathrm{G}$ within the  special complex linear group so that the $\mathrm{QR}$ and the (signed) $\mathrm{LU}$ factorizations for a matrix $X\in \mathrm{G}$ have factors in the subgroup.

 This is always possible \cite{Sa,MT2}. For instance, take $\mathrm{Sp}_\R(2n)$ as the subgroup leaving invariant the non-standard symplectic from  $\sum_{i=1}^n dx_i\wedge dx_{2n+1-i}$ \cite[Table 1]{Sa}.

 Note that the compatibility of the $\mathrm{QR}$ decomposition with $\mathrm{G}$ implies that the Toda vector field on $\sl_\C$ is tangent to the $\mathrm{K}\cap \mathrm{G}$ conjugacy classes of real diagonal matrices in $\gg$.
 \item Verify whether the constrained $\mathrm{QLU}$ factorization for a real diagonal matrix $\Lambda\in \gg$, or a modification of it, returns factors in $\mathrm{G}$.

 We discuss the case of $\mathrm{Sp}_\R(2n)$ for $\Lambda=\{\lambda_1,\dots, \lambda_n,-\lambda_n,\dots,-\lambda_1\}$ which does  not contain zero. Let $\Lambda_1$ denote the $n\times n$ diagonal matrix with the first $n$ eigenvalues. A matrix in the intersection
 $\mathrm{O}_\Lambda\cap \Sp$ is of the form
 \[Q=\begin{pmatrix} Q_1 & 0 \\ 0 & J Q_1 J\end{pmatrix},\quad Q_1\in \mathrm{O}_{\Lambda_1},\]
 where $J$ is the $n\times n$ matrix with 1 in the antidiagonal and 0 elsewhere (the longest permutation). Given $X\in \Sp(2n)$, we write it in block form
 \[X=\begin{pmatrix} X_1 & X_2 \\ X_3 & X_4\end{pmatrix},\]
and apply the constrained $\mathrm{QLU}$ factorization (in $\Sl_\R(n)$) for $\Lambda_1$ to $X_1$,
\[X_1=Q_1L_1U_1,\quad Q_1\in \mathrm{O}_{\Lambda_1},\,L_1\in \L^{\Lambda_1}.\]
We consider the $\L\U$ factorization of the  product $Z=Q^TX\in \Sp(2n)$,
\[Q^TX=\begin{pmatrix} Q_1^T X_1 & Z_2 \\
        Z_3 & Z_4
       \end{pmatrix}=\begin{pmatrix} L_1 & 0 \\ L_3 & L_4\end{pmatrix} \begin{pmatrix} U_1 & U_2 \\ 0 & U_4\end{pmatrix}.
       \]
Define \[\L^\Lambda(\Sp_\R(2n))=\{L\in \L\cap \Sp\,|\,L_1\in \L^{\Lambda^1}\},\]
so that the previous construction yields a unique and smooth factorization
\[X=QLU,\quad Q\in \mathrm{O}_{\Lambda}\cap \Sp(2n), \, L\in \L^\Lambda(\Sp_\R(2n)),\, U\in \U\cap \Sp(2n).\]
This factorization does not agree in general with the constrained $\mathrm{QLU}$ factorization in $\Sl_\R(2n)$ for $\Lambda$ applied to $X\in \Sp_\R(2n)$. Indeed, the subgroup $\L\cap \Sp_\R(2n)$ equals
\[\left\{ L\in \L\,|\ L_4=J{\left(L_1^{-1}\right)}^TJ,\quad L_1^TJL_3=L_3^TJL_1\right\}.\]
It is parametrized by the entries in $L_1$ and the entries in $L_3$ in and above the antidiagonal (the $n^2$ entries in and above the antidiagonal of the $2n \times 2n$ matrix).
If $\Lambda$ has equal  non-contiguous eigenvalues, then there are nontrivial polynomial entries in $L_4$, and this implies that $L^\Lambda(\Sp_\R(2n))\backslash \L^\Lambda\cap \Sp_\R(2n)$ is nonempty. Any $L$ in this subset is factorized by the procedure introduced above as
\[L=\mathrm{I} L\mathrm{I},\]
whereas the constrained $\mathrm{QLU}$ factorization in $\Sl_{\R}(2n)$ for $\Lambda$ returns
\[L=Q'L'U',\]
where neither $Q'$ nor $L'$ are  symplectic,
\[Q'=\begin{pmatrix} \mathrm{I}_n & 0 \\ 0 & Q'_4\end{pmatrix}, \, Q'_4\neq \mathrm{I}_n\,\Longrightarrow\, Q'\notin \Sp_\R(2n),\]
\[L'=\begin{pmatrix} L_1 & 0 \\ L_3' & L_4' \end{pmatrix},\, L'_4\in \L^{J\Lambda_1 J},\, L_4'\neq L_4=J{\left(L_{1}^{-1}\right)}^TJ\,\Longrightarrow\, L'\notin \Sp_\R(2n).\]

We do not know how to adapt the constrained $\mathrm{QLU}$ factorization in $\Sp_\R(2n)$ for those  $\Lambda$ whose spectrum contains 0.

 \item Addressing the non-linearity of $\L^\Lambda(G)$.

 The subgroup $L\cap \G$ is not a subspace in general. For instance, $\L\cap \Sp_\R(2n)$ is a graph over the subspace corresponding to the $n^2$ entries in and above the antidiagonal. In fact, it is the parametrization of this nilpotent group by coordinates of the second kind (under the canonical isomorphism between the subspace and the Lie algebra $\ll\cap \sp_\R(2n)$). Likewise, when $\Lambda$ does not contain 0 in its spectrum, $\L^\Lambda(\Sp_\R(2n))$ is parametrized by coordinates of the second kind restricted to the the intersection of this subspace  with $\L_{\Lambda_1}$.
 This parametrization is equivariant with respect to the action by conjugation by $\mathrm{e}^{t\Lambda}$.  The composition of the parametrization followed by $L\mapsto \kappa(L)^T\Lambda \kappa(L)$ is a chart which diagonalizes the Toda vector field around $\Lambda$ in the symplectic orthogonal conjugacy class of $\Lambda$.

 When $\Lambda$ has simple spectrum, the coordinates diagonalizing the Toda vector field in \cite{MT} replaced the parametrization of the second kind of $\L\cap \Sp_\R(2n)$ by the exponential map. These two sets of diagonalizing coordinates on the (regular) symplectic orthogonal conjugay class are related by the ensuing de Jonquières automorphism of $\ll\cap \sp_\R(2n)$.
\end{enumerate}

\end{document}